\documentclass{article}
\usepackage{graphicx} 

\usepackage{amsmath}
\usepackage{amssymb}
\usepackage{amsthm}
\usepackage{makeidx}
\usepackage{pifont}
\usepackage{bbding}
\usepackage{color}
\usepackage{textcomp}
\usepackage{graphicx}
\usepackage{pgf,tikz}
\usepackage[utf8]{inputenc}
\usepackage{amsmath, amsthm}
\usepackage{amsfonts}
\usepackage{amssymb}
\usepackage{times, fullpage}
\usepackage[T1]{fontenc} 
\usepackage{url} 
\usepackage{color,esint}

\usepackage{hyperref} 
\usepackage{graphicx} 
\usepackage{enumitem}
\usepackage{tabularx}
\usepackage{mathtools}
\date{}

\usepackage[english]{babel}
\usepackage{soul}

\definecolor{vg}{rgb}{0.0, 0.3, 0.15}
  
   \newcommand{\beq}{\begin{equation}}
   \newcommand{\eeq}{\end{equation}}
   \newcommand{\beqs}{\arraycolsep1.5pt\begin{eqnarray}}
   \newcommand{\eeqs}{\end{eqnarray}\arraycolsep5pt}
   \newcommand{\beqsn}{\arraycolsep1.5pt\begin{eqnarray*}}
   \newcommand{\eeqsn}{\end{eqnarray*}\arraycolsep5pt}

  \xdefinecolor{azzurro}{rgb}{0.2,0.9,0.9}
\xdefinecolor{violetto}{rgb}{0.5,0.5,0.9}
\xdefinecolor{marrone}{rgb}{0.8,0.4,0.2}

\def\QED{\hfill $\Box$ }
\newtheorem{pft}{}                              

\newcommand{\bpf}{\begin{pft}\begin{proof}\mbox{\bf Proof}.\ \rm}
\newcommand{\bpfof}[1]{\begin{pft}\begin{proof}{\bf Proof of #1}\
\rm}
\newcommand{\bspf}{\begin{pft}\begin{proof}{\bf Sketch of proof}.\
\rm}

\newcommand{\epf}{\nopagebreak\end{proof}\end{pft}}

\newtheorem{thm}{Theorem}[section]
\newtheorem{rem}[thm]{Remark}
\newtheorem{cor}[thm]{Corollary}
\newtheorem{prop}[thm]{Proposition}
\newtheorem{lemma}[thm]{Lemma}
\newtheorem{defn}[thm]{Definition}
\newtheorem{ex}[thm]{Example}
\numberwithin{equation}{section}

\def \trait (#1) (#2) (#3){\vrule width #1pt height #2pt depth #3pt}
\def \qed{\hfill
        \trait (0.1) (6) (0)
        \trait (6) (0.1) (0)
        \kern-6pt
        \trait (6) (6) (-5.9)
        \trait (0.1) (6) (0)
\medskip}
\def\square{\trait (0.1) (6) (0)
        \trait (6) (0.1) (0)
        \kern-6pt
        \trait (6) (6) (-5.9)
        \trait (0.1) (6) (0)}

\newcommand{\ds}{\rightarrow}

\newcommand{\Om}{\Omega}
\newcommand{\f}{\varphi}

\def\QED{\hbox{\ $\rlap{$\sqcap$}\sqcup$}\bigbreak}

\def\F{{\cal F}}

\def\infess{\mathop{\rm ess\: inf }}
\def\supess{\mathop{\rm ess\: sup }}

\newcommand{\dx}{\,dx}

\newcommand{\wto}{\rightharpoonup}

\newcommand{\loc}{{\rm loc}}

\newcommand{\ainc}[1]{\hyperref[defn:aInc]{{\normalfont(aInc){\ensuremath{_{#1}}}}}}
\newcommand{\adec}[1]{\hyperref[defn:aDec]{{\normalfont(aDec){\ensuremath{_{#1}}}}}}
\newcommand{\inc}[1]{\hyperref[defn:Inc]{{\normalfont(Inc){\ensuremath{_{#1}}}}}}
\newcommand{\dec}[1]{\hyperref[defn:Dec]{{\normalfont(Dec){\ensuremath{_{#1}}}}}}
\newcommand{\azero}{\hyperref[def:a0]{{\normalfont(A0)}}}

\def\supess{\mathop{\rm ess\: sup }}

\font\tenmsb=msbm10 \font\sevenmsb=msbm7 \font\fivemsb=msbm5
\newfam\msbfam
\textfont\msbfam=\tenmsb \scriptfont\msbfam=\sevenmsb
\scriptscriptfont\msbfam=\fivemsb

\newcommand {\Me}{{\cal M}}

\def\to{\rightarrow}
\def\esssup{\mbox{ess sup}}
\def\essinf{\mbox{ess inf}}

\def\d{\, d}

\newenvironment{giacomorev}{\color{blue}}{\color{black}}
\newcommand{\bg}{\begin{giacomorev}}
\newcommand{\eg}{\end{giacomorev}}

\title{Approximation of $L^\infty$  functionals with generalized Orlicz norms}
\author{Giacomo Bertazzoni, Michela Eleuteri, Elvira Zappale}

\AtEndDocument{\bigskip{\footnotesize%
 \textsc{ Giacomo Bertazzoni: Dipartimento di Scienze Fisiche Informatiche e Matematiche, Universit\`{a} degli Studi di Modena e Reggio Emilia, via Campi 213/b, 41125 Modena, Italy} \par  
  \textit{E-mail address}:
\texttt{239975@studenti.unimore.it}
  \par
  \addvspace{\medskipamount}
  \textsc{Michela Eleuteri: Dipartimento di Scienze Fisiche Informatiche e Matematiche, Universit\`{a} degli Studi di Modena e Reggio Emilia, via Campi 213/b, 41125 Modena, Italy} \par  
  \textit{E-mail address}: \texttt{michela.eleuteri@unimore.it} \par
  \addvspace{\medskipamount}
  \textsc{Elvira Zappale: Dipartimento di Scienze di Base e Applicate per l’Ingegneria, Sapienza - Universit\`{a} di Roma,
via Antonio Scarpa, 16, 00161 Roma, Italy} \par
  \textit{E-mail address}: \texttt{elvira.zappale@uniroma1.it}
}}

\date{}

\begin{document}

\maketitle

\begin{abstract}
The aim of this paper is to deal with the asymptotics of generalized Orlicz norms when the lower growth rate tends to infinity. 
$\Gamma$-convergence results and related representation theorems in terms of $L^\infty$ functionals are proven for sequences of generalized Orlicz energies under mild convexity assumptions. 
This latter hypothesis is removed in the variable exponent setting. 


 
\vspace{8pt}
\end{abstract}

\section{Introduction}

In recent years, supremal functionals, i.e. \begin{equation}\label{Fudef}
F(u) := \supess_{x \in \Omega} f(x, u(x), Du(x)),
\end{equation}
with $\Omega \subset \mathbb R^N $ open, $u \in W^{1,\infty}(\Omega; \mathbb R^d)$, and $f : \Omega \times \mathbb R^d \times \mathbb R^{Nd} \rightarrow \mathbb R$ a suitable Borel function,
 have being object of growing interest in view of the many applications. Indeed many variational models are expressed by means of
 energies which are not of integral form since the relevant quantities entering in the modeling
 do not express a mean property.

As a matter of fact, minimization of $L^\infty$ functionals enters in the description of yield set in plasticity, in optimal design, in the description of the dielectric breakdown (see e.g. \cite{BN, CasMS98, GarNP01}), in optical tomography (\cite{K}), in connection with the infinite laplacian as in the seminal papers \cite{A1, A2, A3, A4} and recently in \cite{AK, CKM, RZ0}. Indeed functionals as in \eqref{Fudef} can be seen as the counterpart of integral functionals, which give rise
to the Euler-Lagrange (Euler-Aronsson in this setting) equations (or systems) of $\Delta_{\infty}$-type.
On the other hand, for minimization problems involving \eqref{Fudef}, it is important to
consider the pointwise behavior of the energy density also on very small sets 
 as long as they appear in a very natural way in variational problems, where the  relevant quantities do not express a mean property and the values of the energy densities  on  very small subsets of $\Omega$ cannot be neglected. 

Motivated both by the above mentioned applications of power-law approximation, and by the fact that $L^\infty$-norm of a measurable function is the limit, as $p \to +\infty$ of its $L^p$-norms,
many authors (see e.g. \cite{AP, BMP, BN, BGP, CDP, EP, DEZ, P, PZ}) studied the variational convergence of the energies 
\begin{equation}
\label{Ipdef}
I_p(u) := \left(\int_\Omega f(x, u, Du)^p\, dx\right)^{1/p}
\end{equation}
as $p\to \infty$, obtaining a ($\Gamma$-)limit of supremal type described in terms of a suitable limiting density, not necessarily coinciding with the original one $f$.
More precisely, the limiting density satisfies suitable convexity properties, for instance in the scalar case (i.e. $d$ or $N=1$) it is level convex in the last variable, see Section \ref{conv} for the precise definition.
In particular it coincides with the original one if $f(x,u,\cdot)$ in \eqref{Ipdef}, is convex, quasiconvex in the sense of Morrey, or level convex. 

Still under level convexity assumptions, the energy $I_p$ has been generalized to two complementary cases: 
Orlicz growth by Bocea and Mih\u{a}ilescu \cite{BM2} and variable exponent growth 
by Eleuteri and Prinari \cite{EP}. Moreover in the paper \cite{BHH}, the authors consider 
the same problem in the context of generalized Orlicz spaces, which covers both of these 
as special cases. Generalized Orlicz spaces (also known as Musielak--Orlicz spaces) 
and related PDE have been intensely studied recently, see, e.g., 
\cite{ChlGSW21, HarHL21, HasO22a, HurOS23, JiaWYYZ23} and references 
therein. In fact, both the $L^\infty$- and the non-standard growth energies are related to 
image processing \cite{CasMS98, CheLR06, HarH21}.

 The main idea in \cite{BHH} consists of  generalizing the approach of previous papers by means of improved techniques, with streamlined proofs and weaker assumptions. In particular, assumptions about the lower and upper growth-rates (denoted 
$\phi_n^\pm$ and $p_n^\pm$ in \cite{BM2, EP}) of the approximating functionals are not needed, differently from \cite{BM} and the just mentioned literature. In addition, in \cite{BHH} the authors are able to deal with open sets of finite measure, without requiring boundedness of the sets and  regularity of their boundaries.

In the present paper we aim at weakening the convexity assumptions on the density $f$.
In details, given $\phi\in\Phi_w(\Omega),$ (see Section \ref{tre} for more details on the class $\Phi_w$) and a normal integrand $f: \Omega \times \mathbb R^d \times \mathbb R^{Nd} \to \mathbb R$, 
we define, as in \cite{BHH}, $F_\phi, E_\phi, F_\infty, E_\infty: L^1(\Omega, \mathbb R^d) \to [0,\infty]$  by 
\begin{equation}\label{Fphi}
F_\phi(u) := \begin{cases} 
\|f(\cdot, u, Du)\|_\phi &\text{if } u \in W^{1,1}_\loc(\Omega, \mathbb R^d), \\ 
\infty &\text{otherwise},
\end{cases}
\end{equation}
\begin{equation}\label{Ephi}
E_\phi(u) := \begin{cases} 
\rho_\phi(f(\cdot, u, Du)) &\text{if } u \in W^{1,1}_\loc(\Omega, \mathbb R^d), \\ 
\infty &\text{otherwise},
\end{cases}
\end{equation}
\begin{equation}
\label{Finfinity}
F_\infty(u) = 
\begin{cases}
\|f(\cdot, u, Du)\|_{\infty} &\text{if } u \in W^{1,1}_\loc(\Omega, \mathbb R^d), \\ 
\infty &\text{otherwise}.
\end{cases}  
\end{equation}
and
\begin{equation}
\label{Einfinity}
  E_\infty(u) = 
\begin{cases}
0, &\text{if } u \in W^{1,1}_\loc(\Omega, \mathbb R^d) \text{ and }|f(\cdot, u, Du)|\le 1 \text{ a.e.}, \\
\infty, &\text{otherwise}
\end{cases}  
\end{equation}

More precisely,  we assume that $f$
satisfies the following assumptions:
\begin{enumerate} 
\item [{\bf (H1)}]   $f(x,u,\cdot)$ is ${\rm curl}_{(p > 1)} - Young$ quasiconvex for a.e. $x\in \Omega$ and  for every $u\in \mathbb R^d$, i.e. 
    \[ \supess_{y \in Q}f\left(x, u(x),\int_{\mathbb R^{Nd}} \xi d\nu_y(\xi)\right) \le \supess_{y \in Q} \left(\nu_y - \supess_{\xi \in \mathbb R^{Nd}} f(x, u(x),\xi)\right) 
    \]
    whenever $\nu \equiv \{\nu_y\}_{y \in Q}$ is a $W^{1,p}-$gradient Young measure for every $p \in (1, \infty)$, where $Q$ is the unit cube of $\mathbb R^n$ centered in the origin with side length $1$, (see Section \ref{pre} for details about Young measures and convexity notions);
\item  [{\bf (H2)}] there exist $\alpha, \gamma>0$ such that
\[\ 
f(x,u, \xi) \geq  \alpha |\xi|^{\gamma}  
\qquad \hbox{ for a.e } x\in \Om \hbox{ and for every }
(u, \xi)\in\mathbb R^d\times  \mathbb R^{Nd}.
\]
\end{enumerate} 

We stress that {\bf (H1)} is a milder condition than
 level convexity (i.e. the convexity of the sublevel sets) of $f(x,u,\cdot)$ adopted to prove the results in \cite{EP, BHH}. Indeed condition {\bf(H1)} has been first introduced in \cite[eq. (3.1)]{CDP} to prove $L^p$-approximation for $L^\infty$-type norm functionals, and later characterized in \cite{RZ}, see also \cite{AP, AP2} for similar related notions.
 
 Assuming also that for every $n, \phi_n \in \Phi_w(\Omega)$, satisfies \ainc{{p_n}} with constant $L$ for some $p_n\ge 1$, i.e. \begin{itemize}
\item[{\bf (H3)}]
 there exists $L\ge 1$ such that $\phi_n(x,\lambda t)\le L \lambda^{p_n} \phi_n(x,t)$ 
for a.e.\ $x\in \Omega$ and all $t\ge 0$, $\lambda \le 1$, 
\end{itemize}
and there exists $c >0$ such that 
\begin{itemize}
\item[{\bf (H4)}] $\frac1c \le \phi_n(x,1) \le c$ for a.e.\ $x \in \Omega$,
\end{itemize}
then, letting $F_n:=F_{\phi_n}, E_n:=E_{\phi_n}: L^1(\Omega;\mathbb R^d)\to [0,+\infty]$ be defined by \eqref{Fphi} and \eqref{Ephi}, respectively (with $\phi= \phi_n$) 
we prove in Theorems~\ref{thm:main-norm} that $(F_n)$ sequentially $\Gamma$-converge to $F_\infty$ with respect to the $L^1$-weak topology, as $p_n\to +\infty.$ 
We observe also that if $\Omega$ is bounded with regular boundary, then by Rellich theorem, and ${\bf (H2)}$, guarantee that the weak $L^1$ convergence could be replaced by the strong one and the notion of $\Gamma$-convergence coincides with the one in Proposition \ref{seqcharac}.

\noindent Reinforcing {\bf(H4)} by  
\begin{itemize}
\item[{\bf (H5)}] $\displaystyle \limsup_{n \to \infty} \phi_n^+(1)= 0$ and $\displaystyle \liminf_{n \to \infty} \phi_n^-(1)^{1/p_n}\ge 1,$ 
\end{itemize} 
where $\displaystyle \phi_n^{+}(t):=\supess_{x \in \Omega} \phi_n(x,t)$ and  $\displaystyle \phi_n^{-}(t):={\rm ess} \inf_{x \in \Omega} \phi_n(x, t)$,
\color{black}in Theorem \ref{thm:main-modular} 
we show the (sequential) $\Gamma$-convergence of $(E_n)$  towards $E_\infty$ with respect to the $L^1$-topology.

On the other hand, in Theorems \ref{relnorm1} we are able to remove assumption $\mathbf{(H1)}$, i.e. the ${\rm curl}_{(p >1)}$- Young quasiconvexity assumption on $f:\Omega \times \mathbb R^{Nd}\to \mathbb R$, in its last variable, under the restriction that $f(x,\cdot)$ satisfies a growth condition  of type $|\cdot|^\gamma$, also from above, 
and assuming that $\phi_n(x,t):= t^{p_n(x)}$, with $\sup_{n}\frac{p_n^+}{p_n^-} \leq \beta$, for a $\beta >0$, where $p^+_n = \supess_{x \in \Omega} p(x),\,p^-_n := {\rm ess} \inf_{x \in \Omega} p_n(x)$.

\noindent The result we get gives an extension of \cite[Theorem 2.2]{PZ}, to the case of modulars allowing also to consider $p(\cdot)$-type Sobolev spaces and more general growth conditions. We would like to underline that, differently from \cite{PZ, DEZ}, we deal with Carath\'eodory integrands instead of measurable ones, since the available relaxation results in $W^{1,p}$ and $W^{1,p(\cdot)}$ (crucial for our proofs) for integral functionals with only measurable integrands  (see \cite{Buttazzo, D} and \cite{MM}, respectively)  provide an integral representation in terms of a generic quasiconvex function, which  in principle, in the case $p$ constant, may differ from the one obtained in the case $p$ variable. Hence no comparison among the two densities obtained relaxing $F_{\phi}$ either in $W^{1,p}$ ($\phi(x,\xi)=|\xi|^p$) or in $W^{1, p(\cdot)}$ ($\phi(x,\xi)= |\xi|^{p(\cdot)}$) would be possible. In contrast, in our arguments to obtain a supremal representation by means $L^{p(\cdot)}$-approximation,  we rely in the case $p$ constant, passing from $p(\cdot)$ to a suitable constant exponent $p$.
On the other hand, in this setting, due to our techniques, we have to require a certain regularity on $\Omega$.
\color{black}
\\The paper is organized as follows.
In Section \ref{tre} we recall generalized Orlicz spaces, while results related to useful tools such as Young measures and $\Gamma$-convergence are given in Section \ref{pre}. In Section~ \ref{Lp}, we derive supremal functionals as limits of the generalized Orlicz norms under mild convexity assumptions. We remove in Section \ref{pxgamma} this latter assumption in the case where the Orlicz spaces reduce to Lebesgue spaces with variable exponents.

\smallskip

\section{Generalized Orlicz spaces}\label{tre}

For the purpose of our paper, we consider the case when  $\Omega \subset \mathbb{R}^N$ is  an open set (where $N\geq 1$) and denote by $\mathcal{L}^N(U)$ the $N$-dimensional Lebesgue measure of any measurable set $U$.  In the sequel we consider functions $u: \Omega \rightarrow \mathbb{R}^d$, with $d \ge 1$ and we denote by $k$ any dimension different from $Nd$. \\
In this section, we provide a background on generalized Orlicz spaces. 
For more details, see \cite{ChlGSW21, HarH19}.

\begin{defn}\label{defn:aInc}
Let $f: \Omega\times [0,\infty) \to \mathbb R$ and $p > 0$. We say that $f$ satisfies 
\ainc{p} if there exists $L\ge 1$ such that $f(x,\lambda t)\le L \lambda^p f(x,t)$ 
for a.e.\ $x\in \Omega$ and all $t\ge 0$, $\lambda \le 1$.
\end{defn}

Note that if $\phi$ satisfies \ainc{p_1}, then it satisfies \ainc{p_2}, for every $p_2 < p_1$. 

\begin{defn}
We say that $\phi: \Omega \times [0,\infty) \to [0,\infty]$ is a 
{\emph generalized weak $\Phi$-function} and write $\phi \in \Phi_w(\Omega)$ if 
\begin{itemize}
\item
$x \mapsto \phi(x, |f(x)|)$ is measurable for every $f \in L^0(\Omega)$,
\item
$\phi(x, 0) = 0$, $\displaystyle \lim_{t \to 0^+} \phi(x, t) = 0$ and $\displaystyle \lim_{t \to \infty} \phi(x, t) = \infty$ 
for a.e. $x \in \Omega$, 
\item
$t \mapsto \phi(x, t)$ is increasing for a.e. $x \in \Omega$, 
\item
$\phi$ satisfies \ainc{1}.
\end{itemize}
If $\phi$ does not depend on $x$, then we omit the set and write $\phi \in \Phi_w$.
\end{defn}

We can now define generalized Orlicz spaces. Example~\ref{eg:Linfty} shows that this framework covers 
also $L^\infty$-spaces without the need for special cases. This is crucial here  as we consider the limit when the growth-rate tends to infinity. 

\begin{defn}
Let $\phi \in \Phi_w(\Omega)$ and let the modular $\rho_{\phi}$ be given by
\[ 
\rho_{\phi}(f) := \int_{\Omega} \phi(x, |f(x)|)\,dx
\]
for $f \in L^0(\Omega)$. The set
\[ 
L^{\phi}(\Omega) := \{ f \in L^0(\Omega) : \rho_{\phi}(\lambda f) < \infty \text{ for some } \lambda > 0  \}
\]
is called a {\emph generalized Orlicz space}. It is equipped with the Luxenburg quasinorm
\[ 
\|f\|_\phi 
:= 
\inf\Big\{ \lambda > 0 : \rho_{\phi}\Big(\frac{f}{\lambda}\Big) \le 1\Big\}.
\]
\end{defn} 


\begin{ex}
   \label{eg:Linfty}
Define $\phi_\infty\in \Phi_w(\Omega)$ by $\phi_\infty(x, t):=\infty \chi_{(1,\infty)}(t)$. 
Then $\phi_\infty$ is a generalized weak $\Phi$-function. From the definition 
of modular we see that 
\[
\rho_{\phi_\infty}(f) = 
\begin{cases}
0, &\text{if } |f|\le 1 \text{ a.e.} \\
\infty, &\text{otherwise}.
\end{cases}
\]
It follows from the Luxemburg norming procedure that $\|\cdot\|_{\phi_\infty}=\|\cdot\|_\infty$
and so $L^{\phi_\infty}(\Omega)=L^\infty(\Omega)$.  
\end{ex}

Other examples of generalized Orlicz spaces are (ordinary) Orlicz spaces where $\phi(x,t)$ is independent of $x$, variable exponent spaces $L^{p(x)}$ where $\phi(x, t) = t^{p(x)}$ \cite{DHHR11},  double phase spaces $\phi(x, t) = t^p + a(x) t^q$ \cite{BarCM18, GasP23}, variable exponent double phase spaces $\phi(x, t)= t^{p(x)} + a(x) t^{q(x)}$  \cite{CBGHW22} and many other variants (e.g.\ \cite{BaaB22}). 
The Orlicz--Sobolev space is defined based on $L^\phi(\Omega)$ as usual: 

\begin{defn}
Let $\phi \in \Phi_w(\Omega)$. The function $u \in L^{\phi}(\Omega) \cap W^{1,1}_\loc(\Omega)$ belongs to the {\emph Orlicz--Sobolev space $W^{1,\phi}(\Omega)$}, if the weak partial derivatives $\frac{\partial u}{\partial x_i}$, $i = 1, \dots, N$, belong to $L^{\phi}(\Omega)$. We define a modular and quasinorm on $W^{1, \phi}(\Omega)$ by
\[ 
\rho_{1, \phi}(u): = \rho_\phi(u) + \sum_{i = 1}^N \rho_\phi\bigg(\frac{\partial u}{\partial x_i}\bigg) 
\quad\text{and}\quad
\|u\|_{1,\phi} := \inf\Big\{ \lambda > 0 : \rho_{1,\phi} \Big(\frac{u}{\lambda}\Big) \le 1\Big\}.
\]
\end{defn}

\subsection{Some technical lemmas and crucial examples}
In this subsection we present some useful results needed for the main proofs. For more details the reader can refer to \cite{BHH}. 
The unit ball property is a fundamental relation between quasinorm and modular. 

\begin{lemma}\label{lemma:Unit-ball}
{\rm [Unit ball property, see \cite[Lemma 3.2.3]{HarH19}]}
If $\phi \in \Phi_w(\Omega)$, then
\[ 
\|f\|_{\phi} < 1 
\quad\Rightarrow\quad 
\rho_{\phi}(f) \le 1 
\quad\Rightarrow\quad 
\|f\|_{\phi} \le 1.
\]
\end{lemma}

This next property anchors $\phi$ at $1$. 

\begin{defn}
We say that a  function $\phi \in \Phi_w(\Omega)$ satisfies \azero \, if there exists $\beta \in (0,1]$ such that $\phi(x,\beta) \le 1 \le \phi(x, \frac{1}{\beta})$ for a.e.\ $x \in \Omega$.
\end{defn}
It follows from \ainc{1}  that if $\frac1c \le \phi(x, 1) \le c$ for a.e.\ $x \in \Omega$, 
then $\phi$ satisfies \azero \,  with $\beta := \frac1{Lc}$.\\
The following proposition plays a significant role in reducing our main problem to 
the classic Lebesgue space $L^q(\Omega)$, where we can use the results about Young measures previously presented.
The embedding is well-known (see \cite{HarH19}) but we need a version with an asymptotically sharp constant as $p\to\infty$.

\begin{prop}
\label{prop:embedding}
Let $\Omega$ have finite measure, $\phi\in \Phi_w(\Omega)$ and $c, L\ge 1$. Assume  that 
\begin{itemize}
\item $\frac1c \le \phi(x,1)\le c$ 
for a.e.\ $x \in \Omega$;
\item $\phi$ satisfies \ainc{p} with constant $L$ for some $p\ge 1$.
\end{itemize}
Then $L^{\phi}(\Omega) \hookrightarrow L^p(\Omega)$ and 
\begin{equation}
\label{Prop3.8}
\|\cdot\|_{L^p} \le (2L(|\Omega|+c))^\frac1p \|\cdot\|_{\phi}.
\end{equation}
\end{prop}

The next proposition, (cf. \cite[Proposition 3.4]{BHH} also for the proof) allows us to prove the  $\limsup$-part of the estimate in 
the main results. The almost increasing assumption is satisfied (with $L=1$) 
for example when $\phi_n^{1/p_n}$ is convex.

\begin{prop}
\label{prop}
Let $\Omega$ have finite measure, 
$\phi_n\in \Phi_w(\Omega)$ for $n\in\mathbb N$ and $c, L\ge 1$. Assume for $n\in\mathbb N$ that 
\begin{itemize}
\item $\frac1c \le \phi_n(x,1) \le c$ for a.e.\ $x \in \Omega$;
\item $\phi_n$ satisfies \ainc{p_n} with constant $L$ for some $p_n\ge 1$.
\end{itemize}
If $p_n\to \infty$, then $\displaystyle \lim_{n \to \infty} \|u\|_{\phi_n} = \|u\|_{\infty}$ for all $u \in L^{\infty}(\Omega)$.
\end{prop}

An example which demonstrates that \azero \,
 is not sufficient to prove  our main results (the assumption $\frac1c \le \phi_n(x,1) \le c$  is stronger, indeed) can be found in \cite[Example 3.5]{BHH}.
\\
A second example  which shows that 
the assumption that each $\phi_n$ satisfy \ainc{p_n} with the same constant $L$ is also needed for the conclusion, can be found in \cite[Example 3.6]{BHH}.
Note, also, that both examples are of Orlicz type as the dependence of $\phi$ on $x$ is not used here.

\color{black}

\subsection{Variable exponents Lebesgue-Sobolev spaces}

\color{black}
In  this subection we consider the variable exponent, as a special case  of generalized Orlicz, with $\phi(x, t) = t^{p(x)}$ for some measurable function $p: \Omega \rightarrow [1, + \infty],$  called {\it variable exponent} on $\Omega$. We collect some more specific results concerning variable exponent Lebesgue and Sobolev spaces. For more details we refer to the monograph \cite{DHHR11}.
\begin{defn}
For any (Lebesgue) measurable function $p: \Omega \rightarrow [1, + \infty]$
we define
$$
p^- := \displaystyle \infess_{x \in \Omega} p(x) \qquad \qquad p^+ :=\displaystyle
\supess_{x \in \Omega} p(x).
$$ 

\end{defn}


 If $ p^+<+\infty$ we call $p$ a {\it bounded variable exponent}, and we can define the variable exponent Lebesgue space as \[
L^{p(\cdot)}(\Omega) := \left \{u: \Omega \rightarrow \mathbb R \,\,\, \textnormal{measurable such that} \,\, \int_{\Omega} |u(x)|^{p(x)} \, dx < + \infty   \right \},
\]

which is a Banach space endowed with the {\it Luxemburg norm}.

By 
Corollary 3.3.4 in \cite{DHHR11}, if  $0 <  \mathcal{L}^N(\Omega) < + \infty$ and $p$ and $q$ are variable exponents such that $p \le q$   a.e. in $\Omega$, then the embedding $L^{q(\cdot)}(\Omega) \hookrightarrow L^{p(\cdot)}(\Omega)$ is continuous.  The embedding constant is less or equal to  
$2( \mathcal{L}^N(\Omega)  + 1)$  and 
$2  \max  \left \{  {\mathcal{L}^N(\Omega)}^{{(\frac 1 q-\frac 1 p)}^+},  {\mathcal{L}^N(\Omega)}^{{(\frac 1 q-\frac 1 p)}^-}  \right \}.$

For any variable exponent $p$, we define $p'$ by setting
\[
\frac{1}{p(x)} + \frac{1}{p'(x)} = 1,
\]
with the convention that, if $p(x) =+ \infty$ then $p'(x) = 1$. The function $p$ is called {\it the dual variable exponent of $p$}.

We have the following result (for more details see Lemma 3.2.20 in \cite{DHHR11}).

\begin{thm} {\sl (H\"older's inequality)}
Let $p,q,s$ be measurable exponents such that
\[
\frac{1}{s(x)} = \frac{1}{p(x)} + \frac{1}{q(x)}
\]
a.e. in $\Omega$. Then
\[
\|fg\|_{s(\cdot)} \le \, \left( \left(\frac{s}p\right)^+ + \left(\frac{s}q\right)^+ \right)\, \|f\|_{p(\cdot)} \, \|g\|_{q(\cdot)}
\]
for all $f \in L^{p(\cdot)}(\Omega)$ and $g \in L^{q(\cdot)}(\Omega)$, where in the case $s = p = q = \infty$, we use the convention $\frac{s}{p} = \frac{s}{q} = 1.$
\\
In particular, in the case $s  = 1$, we have
\[
\left | \int_{\Omega} f \, g \, dx\right | \le \, \int_{\Omega} |f| \, |g| \, dx \le {\, \left( \frac{1}{p^-} + \frac{1}{p'^-} \right)}\, \, \|f\|_{p(\cdot)} \, \|g\|_{p'(\cdot)}.
\] 
\end{thm}

The {\it modular} in this setting specializes as, 
\[
\rho_{p(\cdot)}(u) := \int_{\Omega} |u(x)|^{p(x)} \,dx.
\]
In particular, when $p^+<+\infty$,
by Lemma 3.2.5 in \cite{DHHR11},    for every $u \in L^{p(\cdot)}(\Omega)$ it holds
\begin{equation} \label{relaztotale}\min \left\{\big( \rho_{p(\cdot)}(u)\big)^{\frac 1 {p^-} } , \big(\rho_{p(\cdot)}(u)\big)^{\frac 1 {p^+} }  \right \} \leq  \|u\|_{p(\cdot)} \leq \max  \left \{  \big( \rho_{p(\cdot)}(u)\big)^{\frac 1 {p^-} } , \big(\rho_{p(\cdot)}(u)\big)^{\frac 1 {p^+} } \right \}.
\end{equation}

%

%
 We conclude this part by variable exponent Sobolev spaces, referring, for more details to \cite{DHHR11}, Definition 8.1.2). 

\begin{defn}
 
Let $k,d\in \mathbb N$, $k\geq 0$,  and let  $p$ be a measurable exponent. We define
$$W^{k, p(\cdot)}(\Omega,\mathbb R^d):=\{ u:\Omega\to \mathbb R^d :u, \partial_{\alpha} u \in L^{p(\cdot)}(\Omega,\mathbb R^d) \quad \forall \alpha \hbox{  multi-index such that $|\alpha| \le \, k$ }  \},$$
where $$L^{p(\cdot)}(\Omega,\mathbb R^d):=\{ u:\Omega\to \mathbb R^d :\ |u| \in L^{p(\cdot)}(\Omega)  \}.$$
We define the semimodular on $W^{k, p(\cdot)}(\Omega)$ by
\[
\rho_{W^{k, p(\cdot)}(\Omega)}(u) := \sum_{0 \le |\alpha| \le k} \rho_{L^{p(\cdot)}(\Omega)}(|\partial_{\alpha} u|)
\]
which induces a norm  by 
\[
\|u\|_{W^{k, p(\cdot)}(\Omega)} := \inf \left \{ \lambda > 0: \,\, \rho_{W^{k, p(\cdot)}(\Omega)}\left (\frac{u}{\lambda} \right ) \le \, 1 \right \}.
\]
\end{defn}
 Clearly $W^{0, p(\cdot)}(\Omega) = L^{p(\cdot)}(\Omega)$.

\color{black}
\section{Preliminary results}\label{pre}
In the sequel we present some preliminary definitions, properties and results mainly dealing with Young measures, $\Gamma$-convergence and convexity notions in the supremal setting. 
\subsection{Young measures}
In this subsection we briefly recall some results on the theory of Young measures (see e.g. \cite{Ba89}, \cite{BL}). Let $\Omega\subseteq \mathbb R^N$ be (as above) an open set (not necessarily
bounded) and  $d\geq 1$, we denote by $C_c(\Omega;\mathbb R^d)$ the set of continuous functions with compact support in $\Omega$, endowed with the supremum
norm. The dual of the closure of $C_c(\Omega;\mathbb R^d)$  may be identified
with the set of $\mathbb{R}^d$-valued Radon measures with finite mass
$\Me(\Omega;\mathbb{R}^d)$, through the duality
$$
\langle \mu, \varphi\rangle := \int_\Omega \varphi(\xi)\d\mu (\xi)\,, \qquad
\mu\in \Me(\Omega;\mathbb{R}^d)\,,\qquad \varphi\in C_c(\Omega;\mathbb{R}^d)\,.
$$

\begin{defn} A map $\mu:\Omega\mapsto \Me(\Omega;\mathbb{R}^d)$ is said to be
weak$^*$-measurable if $x\mapsto \langle \mu_x, \varphi\rangle$ are
measurable for all $\varphi\in C_c(\Omega;\mathbb{R}^d)$.
\end{defn}


The main result concerning Young Measures is the so called {\sl Fundamental Theorem on Young Measures}, which we omit, referring to \cite[Theorem 3.1]{Mu} where also a proof can be found.


The map $\mu:U\mapsto \Me(\mathbb R^d)$ is called {\it Young measure generated by the sequence} $(V_n)$ if 
$(V_n)$ is a sequence of measurable functions, $V_n: U\mapsto
\mathbb R^d$ and there exists a subsequence $(V_{n_k})$ and a weak$^*$-measurable map $\mu:U\mapsto \Me(\Omega; \mathbb R^d)$ such that the following  holds: 
 $\mu_x\ge 0$,
$\displaystyle{\Vert\mu_x\Vert_{\Me(\Omega; \mathbb R^d)}=\int_{\mathbb R^d}d\mu_x\le 1}$ for a.e. $x\in
U$.
\color{black}

%
%

 

 In particular, if  $(V_n)$   an equi-integrable sequence in  $L^1(U,\mathbb R^d)$, as a consequence of the Fundamental Theorem of Young measures, it generates the Young measure $\mu=(\mu_x)$ satisfying   $\Vert\mu_x\Vert_{\Me(\mathbb R^d)}=1$  such that 
$ V_{n_k} \wto  \bar{V} \hbox{ weakly in } L^1(U,\mathbb R^d)\,
$, where $$
\bar{V}(x)=  \int_{\mathbb R^d} \xi \d
\mu_x(\xi) \qquad  {\rm for\: a.e.}\: x\in U\,.
$$

The following Corollary \ref{2var} allows us to treat limits of  integrals in the form $\displaystyle \int_U f(x, V_{n}(x), DV_{n}(x))dx$ without any convexity assumption of $f(x,u,\cdot)$ (see Corollary 3.3 in \cite{Mu}).
\noindent First we recall the following definitions. 

\begin{prop}
\label{2var}
Let $U \subset \mathbb R^N$ be a measurable set with finite measure. Assume that the sequence of measurable functions $V_{n}: U\mapsto
\mathbb R^d$ generates the Young measure $(\mu_x)$. 
\begin{enumerate} 
\item If $f:U\times \mathbb R^d\mapsto \mathbb{R}^d$ is a  normal integrand such that  the negative part $f(x, V_{n}(x))^-$ is weakly relatively compact in $L^1(U,\mathbb R^d),$ then 
$$
\liminf_{n\to \infty} \int_U f(x, V_{n}(x))\dx \ge
\int_U \bar{f}(x)\dx\,,
$$

where
$$
\bar{f}(x):= \langle \mu_x , f(x,\cdot)\rangle = \int_{\mathbb R^d} f(x,y)\d
\mu_x(y)\,;
$$
\item if $f$ is a Carath\'eodory integrand such that  $(|f(\cdot, V_{n}(\cdot))|)$ is equi-integrable, then
$$
\lim_{n\to \infty} \int_U f(x, V_{n}(x))\dx = \int_U
\bar{f}(x)\dx<+\infty.
$$
\end{enumerate}
\end{prop}

\begin{rem}
{\rm  If $p>1$ and $\Omega\subseteq \mathbb R^N$ is a bounded open set and  $u_n\wto u$ in $W^{1,p}(\Omega,\mathbb R^d)$,   then the sequence  $(Du_n)$ is equi-integrable and generates a Young measure  $\mu=(\mu_x)$ such that $\Vert\mu_x\Vert_{\Me(\mathbb R^d)}=1$ for a.e. $x\in \Omega$  and 
$$ Du(x)= \int_{\mathbb R^{Nd}} \xi  \d
\mu_x(\xi)\,.
$$

Such a Young measure $\mu$ is usually called a $W^{1,p}$-{\sl gradient Young measure}, see \cite{Pe}.\\
Moreover, by Corollary 3.4 in \cite{Mu},  the couple $(u_n,Du_n)$ generates the Young measure $x\to \delta_{u(x)}\otimes \mu(x)$, and, if $f:\Om\times\mathbb R^d\times\mathbb R^{Nd}\to\mathbb R$ is a normal integrand bounded from below then, by Proposition \ref{2var} (1),  it follows that
\[ 
\liminf_{n\to \infty} \int_\Omega f(x, u_n(x),Du_{n}(x))\dx \geq  \int_\Omega\int_{\mathbb R^{Nd}}
f(x,u(x),\xi)d\mu_{x}(\xi)   \dx.
\]
}\end{rem}
\subsection{Convexity notions for supremal functionals}
\label{conv}

We are in position to recall the  convexity notions which are key in the supremal setting, referring to \cite{RZ} for a general review.

The following notion is proved to be necessary and sufficient for the lower semicontinuity of supremal functionals in the scalar setting, \cite{ABP, BJW} and sufficient for $L^p$ and $L^{p(\cdot)}$- approximation, see \cite{EP, PZ}, and the references therein.

\begin{defn}\label{deflc}We say that $f:\mathbb R^k\to \mathbb R$ is level convex if
 for every $t\in\mathbb R$ the level set $\big\{\xi\in\mathbb R^k\colon f(\xi)\le t\big\}$ is convex. 
\end{defn}

On the other hand such condition is not necessary for the lower semicontinuity of supremal functionals in the vectorial setting as firstly shown in \cite{BJW}.
Indeed, in this latter paper the authors characterized the lower semicontinuity of $L^\infty$-functionals by the so-called strong-Morrey quasiconvexity, which we recall next and rename, as in \cite{RZ}, {\it BJW quasi-level-convexity}.

\begin{defn}
     A function $f:\mathbb R^{Nd}\to \mathbb R$ is said to be strong Morrey quasiconvex ({\it BJW quasi-level-convex})if $$\forall\ \varepsilon>0\ \forall\ \xi\in \mathbb{R}^{Nd}\ \forall\ K>0\ \exists\ \delta=\delta(\varepsilon, K,\xi)>0:$$
$$\left.\begin{array}{l}\varphi\in W^{1,\infty}(Q;\mathbb{R}^d)\vspace{0.2cm}\\ ||D\varphi||_{L^\infty(Q;\mathbb{R}^{Nd})}\le K\vspace{0.2cm}\\ \max_{x\in\partial Q}|\varphi(x)|\le \delta\end{array}\right\}\Longrightarrow f(\xi)\le \operatorname*{ess\,sup}_{x\in Q}  f(\xi+D\varphi(x))+\varepsilon.$$
\end{defn}

On the other hand the relaxation in the supremal setting is currently open, as well as the understanding of the sufficiency of {\it BJW -quasi level convexity} (strong Morrey quasiconvexity) for $L^p$-approximation,  hence several other notions have been introduced in the literature \cite{AP, AP2, CDP, RZ}.



In particular,
we recall the notion of ${\rm curl}_{p\geq 1}$-Young quasiconvexity as in \cite{RZ}, which generalizes the concept of level-convexity (see Definition \ref{deflc}) and reduces to it in the scalar setting. We refer to \cite{CDP, AP} for previously introduced related notions and to \cite{RZ} for a comparison.
\begin{defn}\label{RZdef}
    Assume that $f$ is lower semicontinuous and bounded from below. $f$ is said to be ${\rm curl}_{(p > 1)}-Young$ quasiconvex, if
    \[ \supess_{x \in Q}f\left(\int_{\mathbb R^{Nd}} \xi d\nu_x(\xi)\right) \le \supess_{x \in Q} \left(\nu_x - \supess_{\xi \in \mathbb R^{Nd}} f(\xi)\right) 
    \]
    whenever $\nu \equiv \{\nu_x\}_{x \in Q}$ is a $W^{1,p}-$gradient Young measure for every $p \in (1, \infty)$, where $Q$ is the unit cube of $\mathbb R^n$ centered in the origin with side length $1$.
\end{defn}

The following result, which provides characterizations of ${\rm curl}_{(p>1)}$- Young quasiconvexity has been proven in \cite{RZ}.
\begin{prop}\label{RZthm}
    Let $f: \mathbb R^{Nd} \rightarrow \mathbb R$ be a lower semicontinuous function and bounded from below. Then the following conditions are equivalent.
    \begin{itemize}
        \item[$(i)$] $f$ is ${\rm curl}_{(p > 1)}-Young$ quasiconvex;
        \item[$(ii)$] $f$ verifies
        \[ f\left(\int_{\mathbb R^{Nd}} \xi d\nu_x(\xi)\right) \le \nu_x - \supess_{\xi \in \mathbb R^{Nd}} f(\xi),\,\,\text{ for a.e. } x \in Q \]
    whenever $\nu \equiv \{\nu_x\}_{x \in Q}$ is a $W^{1,p}-$gradient Young measure for every $p \in (1, \infty)$ and $Q$ is the unit cube of $\mathbb R^n$ centered in the origin with side length $1$; 
    \item[$(iii)$] $f$ verifies
        \[ f\left(\int_{\mathbb R^{Nd}} \xi d\nu_x(\xi)\right) \le \nu - \supess_{\xi \in \mathbb R^{Nd}} f(\xi),\]
    whenever $\nu$ is a homogeneous $W^{1,p}-$gradient Young measure for every $p \in (1, \infty)$.
    \end{itemize}
    Moreover, in the definition of ${\rm curl}_{(p > 1)}-Young$ quasiconvexity the domain $Q$ can be replaced by any open, bounded, connected set $\Omega \subset \mathbb R^N$ with Lipschitz boundary.
\end{prop}

In the next section we will prove that ${\rm curl}_{(p > 1)}-Young$ quasiconvexity is sufficient for power law approximation also in the generalized Orlicz-Sobolev setting. 

\subsection{$\Gamma-$convergence}

We recall the sequential characterization of  the $\Gamma$-limit when $X$ is a metric space. For more details we refer to \cite{B, DM93}.

\begin{prop}\label{seqcharac}[\cite[Proposition 8.1]{DM93}]
Let $X$ be a metric  space and let $\f_n: X \ds \mathbb R \cup
\{\pm \infty\}$ for every $n\in \mathbb N$.
 Then  $(\f_n)$ $\Gamma$-converges to $\f$ with respect to the strong topology of  $X$ (and we write $ \Gamma(X)$-$\lim_{n\to \infty}\f_n=\f$)  if and only if
\begin{description}
\item [(i)] {\rm($\Gamma$-$\liminf$ inequality)} for every $x\in X$ and for every sequence $(x_n)$ converging
to $x$, it is
$$  \f(x)\le \liminf_{n\to \infty} \f_n(x_n);$$
\item [(ii)]{\rm ($\Gamma$-$\limsup$ inequality)} for every $x \in X$, there exists a sequence $(x_n)$  converging
to $x$ such that
$$  \f(x)=\lim_{n\to \infty} \f_n(x_n).$$
\end{description}
\end{prop}
We recall that the  $\Gamma$-$\lim_{n\to \infty}\f_n$ is lower semicontinuous on $X$ (see \cite[Proposition 6.8]{DM93}).

Analogous definition can be provided for the $\Gamma$-convergence with respect to the waek * topology in X, due to its metrizability in compact sets, denoted by $\Gamma(w^*\hbox{-}X)\hbox{-}\lim_{n\to \infty}\f_n$ .
Hence we recall also  that the function  $\f=\Gamma(w^*\hbox{-}X)\hbox{-}\lim_{n\to \infty}\f_n$ is weakly* lower semicontinuous  on $X$ (see \cite{DM93} Proposition 6.8) and when  $\f_n=\psi$  $\forall n\in\mathbb N$ then   $\f$   coincides with the  weakly*  lower semicontinuous (l.s.c.) envelope  of $\psi$, i.e. 
\[ 
\f(x)=\sup\big\{h(x):\  \forall\, h:X\to\mathbb R\cup\{\pm\infty\} \   \ w^* \hbox
{ l.s.c.}, \  h\le \psi\hbox { on } X\big\}
\]
(see Remark 4.5 in \cite{DM93}).
\\
Finally we recall that (i) and (ii) can be assumed to be definition for the sequential $\Gamma$-limit of a sequence of functionals $(\f_n)$ also if $X$ is not a metric space, for instance in the case of weak topology in some function space. With an abuse of notation we will not distinguish between sequential $\Gamma$-limits and $\Gamma$-limits as it will be clear from the context (see \cite{DM93} for details). In particular we observe that the results of Theorems \ref{thm:main-norm} and \ref{thm:main-modular} are in terms of sequential $\Gamma$-convergence.

We will say that a family $(\f_p)$ $\Gamma$-converges to $\f$, with respect to the topology considered on $X$ as $p\to \infty$,
if $(\f_{p_n})$ $\Gamma$-converges to $\f$ for all sequences $(p_n)$ of positive numbers diverging to $\infty$ as $n\to\infty$.

\smallskip

\noindent For details about $\Gamma$-convergence we refer to \cite{DM93} and \cite{B}.

\section{Lower growth approximation}\label{Lp}

In this section we study the approximation as the lower growth rate tends to $\infty$, via $\Gamma$-convergence, of supremal functionals in terms of generalized Orlicz-type quasinorms.

In the following we consider a normal integrand  $f:\Om\times\mathbb R^d\times\mathbb R^{Nd}\to\mathbb R$, satisfying {\bf (H1)}, {\bf(H2)}.
\subsection{Statement of the main results}


We start by stating all theorems to easily compare the  results obtained according to the different set of  hypotheses and topologies considered.


Let $\phi\in\Phi_w(\Omega)$, 
we recall the functionals $F_\phi, E_\phi : L^1(\Omega, \mathbb R^d) \to [0,\infty]$, introduced in \eqref{Fphi} and \eqref{Ephi}, defined by 
\begin{equation*}
F_\phi(u) := \begin{cases} 
\|f(\cdot, u, Du)\|_\phi &\text{if } u \in W^{1,1}_\loc(\Omega, \mathbb R^d), \\ 
\infty &\text{otherwise}.
\end{cases}
\end{equation*}
and
\begin{equation*}
E_\phi(u) := \begin{cases} 
\rho_\phi(f(\cdot, u, Du)) &\text{if } u \in W^{1,1}_\loc(\Omega, \mathbb R^d), \\ 
\infty &\text{otherwise},
\end{cases}
\end{equation*}
The functionals $F_\infty$ and $E_\infty$ in \eqref{Finfinity} and \eqref{Einfinity} refer to the case
$\phi=\phi_\infty$ from Example~\ref{eg:Linfty}, related to $L^\infty$, i.e.
\begin{equation*}
  F_\infty(u) = 
\begin{cases}
\|f(\cdot, u, Du)\|_{\infty} &\text{if } u \in W^{1,1}_\loc(\Omega, \mathbb R^d), \\ 
\infty &\text{otherwise}.
\end{cases}  
\end{equation*}
and
\begin{equation*}
  E_\infty(u) = 
\begin{cases}
0, &\text{if } u \in W^{1,1}_\loc(\Omega, \mathbb R^d) \text{ and }|f(\cdot, u, Du)|\le 1 \text{ a.e.}, \\
\infty, &\text{otherwise},
\end{cases}  
\end{equation*}
\color{black}

Let $(F_n)$ and $(E_n)$ defined by \eqref{Fphi} and \eqref{Ephi}, respectively, as $F_n:= F_{\phi_n}$ and $E_n:=E_{\phi_n}$. We prove in Theorems~\ref{thm:main-norm} and \ref{thm:main-modular} 
that $(F_n)$ and $(E_n)$ (sequentially) $\Gamma$-converge to $F_\infty$ and $E_\infty$ respectively, with respect to the $L^1$-weak topology. Specifically, we prove the 
$\liminf$-property of $\Gamma$-convergence, i.e. (i) of Proposition \ref{seqcharac} and show that we can use a constant recovery 
sequence for (ii) in Proposition \ref{seqcharac}.


\begin{thm}\label{thm:main-norm}
Assume that $\Omega$ has finite measure. Let 
$f: \Omega \times \mathbb R^d \times \mathbb R^{Nd} \to \mathbb R$ be a normal integrand such that  ${\bf (H1)}$ and ${\bf (H2)}$ hold.

Let $\phi_n\in \Phi_w(\Omega)$ for $n\in\mathbb N$ and $c, L\ge 1$ and assume that {\bf (H3)} and {\bf (H4)}
hold. 

If $p_n \to \infty$ as $n \to \infty$, then
\[ 
\limsup _{n \to \infty} F_n(u) \le  F_\infty(u) 
\le \liminf _{n \to \infty} F_n(u_n)
\]
for all $u, u_n \in L^1(\Omega, \mathbb R^d)$ with $u_n \rightharpoonup u$ in $L^1(\Omega, \mathbb R^d)$.
\end{thm}

\begin{cor}\label{gamma2} 
Let $X\in \{ L^{\infty}(\Omega,\mathbb R^{d}), C(\Omega,\mathbb R^{d})\}$ be endowed with the norm $||\cdot||_{\infty}$. Under the same assumptions of Theorem \ref{thm:main-norm}, let $F_n, F_\infty :X\to [0, + \infty]$ be the functionals  defined by \eqref{Fphi} (with $\phi= \phi_n$) and \eqref{Finfinity}, respectively.
Then the sequence $(F_n)$  $\Gamma\hbox{-}$ converges, with respect to the $L^{\infty}$-strong  convergence, to the functional $F_\infty$. 
 as $n\to +\infty$.
 \end{cor}

We conclude with similar results for the modular-based energy functional. We use the notation $\displaystyle\phi^-(1):= \essinf_{x \in \Omega} \phi(x, 1)$ and $\displaystyle\phi^+(1):= \esssup_{x \in \Omega} \phi(x, 1)$. Arguing as in the proof of \cite[Theorem 4.4]{BHH} and exploiting the previous theorem, the following corollary holds.
\begin{thm}\label{thm:main-modular} Let $\Omega$, $f, \phi_n$ and $L$ as in Theorem \ref{thm:main-norm}, satisfying ${\bf (H1)}-{\bf (H3)}$ and {\bf (H5)} 

Let $E_n$ be the functional defined by \eqref{Ephi} (with $\phi=\phi_n$).
If $p_n \to \infty$ as $n \to \infty$, then
\[ 
\limsup _{n \to \infty} E_n(u) \le  E_\infty(u) 
\le \liminf _{n \to \infty} E_n(u_n)
\]
for all $u, u_n \in L^1(\Omega, \mathbb R^d)$ with $u_n \rightharpoonup u$ in $L^1(\Omega, \mathbb R^d)$.
where $E_\infty$ is the functional in \eqref{Einfinity}.
\end{thm}

\begin{cor}\label{gamma6} Let $X$ be as in Corollary \ref{gamma2}. Under the same assumptions of Theorem \ref{thm:main-modular}. 
let $E_n, E_\infty:X\to [0, + \infty]$ be the functionals defined by equations \eqref{Ephi} (with $\phi=\phi_n$) 
and \eqref{Einfinity}, respectively.
Then, $(E_n)$  $\Gamma\hbox{-}$ converges $E_\infty$, as $n\to +\infty$,  with
 respect to the $L^{\infty}$- strong  convergence.
 \end{cor}

\subsection{Proofs of {Theorems}}
 
 {\bf Proof of Theorem \ref{thm:main-norm}}
 The proof is very similar to the one of \cite[Theorem 4.2]{BHH}. The main steps are written for the readers' convenience.

Let $F_n:= F_{\phi_n}$, $n\in\mathbb N$, be the functional in  \eqref{Fphi}, with $\phi=\phi_n$.
To show that
\[
\limsup_{n \to \infty} F_n(u) \le F_\infty(u)
\]
for $u \in L^1(\Omega, \mathbb R^d)$, we assume, without loss of generality, that $F_\infty(u) <+\infty$, 
hence, by Proposition \ref{prop}, it follows that 
\[ 
\lim_{n \to \infty} F_n(u) = \lim_{n \to \infty} \|f(\cdot, u, Du)\|_{\phi_n} = \|f(\cdot, u, Du)\|_{\infty} = F_\infty(u).
\]
We next deal with the $\liminf$-inequality, following line by line \cite[Theorem 4.2]{BHH}, hence we can assume without loss of generality that $\Omega$ is a ball $B$.  
Let $(u_n) \subset L^1(B, \mathbb R^d)$ 
converge weakly to $u \in L^1(B, \mathbb R^d)$.  Without loss of generality, we assume that 
\[ 
\liminf_{n \to \infty} F_n(u_n) = \lim_{n \to \infty} F_n(u_n) = M < +\infty.
\]

Then every subsequence of $(u_n)$ also has limit  $M$ in energy. We will prove that it will converge in any $L^q$ for every $q>\gamma$.
Recall that $p_n\to\infty$. Fix such a $q$ and let $n_0 \in \mathbb N$ be such that
\[
p_n \ge \frac q\gamma \quad\text{and}\quad F_n(u_n) \le M + 1\qquad\text{for all } n \ge n_0.
\]
From $p_n \ge \frac q\gamma$ it follows that $\phi_n$ satisfies \ainc{\frac{q}{\gamma}} with 
the same constant $L$ as in the assumption. By Proposition~\ref{prop:embedding},
\begin{equation*}
\|f(\cdot, u_n, Du_n)\|_{q/\gamma}  
\le  
\underbrace{\big(2L(|\Omega|+c)\big)^\frac\gamma q}_{=:C_q} 
\|f(\cdot, u_n, Du_n)\|_{\phi_n} 
\le
C_1 (M+1).
\end{equation*}
for every $n \ge n_0$, and  $C_q\searrow 1$ as $q\to \infty$.

By this inequality and  ${\bf (H2)}$, it follows:
\[
\|Du_n\|^{\gamma}_{q} 
= 
\big\||Du_n|^\gamma\big\|_{q/\gamma} 
\le \tfrac{1}{\alpha}\|f(\cdot, u_n, Du_n)\|_{q/\gamma}
\le \tfrac{C_1}{\alpha}  (M+1),
\]
hence $(\|Du_n\|_q )$ is bounded in $n$ for every $q\geq \gamma$.
Then, up to a subsequence (possibly depending on $q$), $(Du_n)$ weakly converges to a function $w$ in $L^q(B, \mathbb R^{Nd})$. 
Since $(u_n)$ weakly converges to $u$ in $L^1(B, \mathbb R^d)$, $w$ is the distributional gradient of $u$. 
%
By Poincar\'e inequality where $L^q$-integrability is required only for the gradient, cf. \cite[Section 1.5.2, p. 35]{MazP97},
we obtain that 
\begin{align*}
\|u_n\|_{L^q(B, \mathbb R^d)} 
\le 
c(N, B) \|Du_n\|_{L^q(B, \mathbb R^d)} + |B|^{\frac1q-1} \|u_n\|_{L^1(B, \mathbb R^d)},
\end{align*}
Consequently $(u_n)$ has a weakly convergent subsequence in $W^{1,q}(B, \mathbb R^{d})$.  

The metrizability of weak topology  $W^{1,q}(B;\mathbb R^d)$ in compact sets allows us to apply Urysohn property, hence to  conclude that
the entire sequence $(u_n)$ is such that $ u_n \rightharpoonup u$ in $W^{1,q}(B;\mathbb R^d)$. 
Let $(\nu_x)_{x \in B}$ be the family of $W^{1,q}$-gradient Young measures generated by $(D u_n)$ for every $q>\min\{\gamma, 1\}$ 
(with barycenter $Du(x)$). Arguing as in \cite[Theorem 4.2]{BHH}
\begin{align*}
\liminf_{n \to \infty} F_n(u_n) &
\ge \liminf_{q \to \infty}\bigg( \int_B \int_{\mathbb R^{Nd}} f(x, u(x), \xi)^{\frac q\gamma}\, d\nu_x(\xi)\, dx\bigg)^{\frac{\gamma}{q}}\\
&\ge \esssup_{x \in B} f(x, u(x), Du(x)),
\end{align*}
where in the last inequality we have exploited the fact that $f(x,u,\cdot)$ is ${\rm curl}_{(p > 1)}-Young$ quasiconvex for a.e.\ $x \in \Omega$ and every $u \in \mathbb R^d$ and (ii) in Theorem \ref{RZthm}, taking into account that 
\begin{equation}\label{barycenter}
Du (x)= \int_{\mathbb R^{Nd}}\xi d \nu_{x}(\xi).
\end{equation}

Thus we have 
completed the proof.  
\qed

\noindent {\bf Proof of Corollary \ref{gamma2}} As in the proof of Theorem \ref{thm:main-norm}, the $\Gamma$-$\limsup$ inequality  follows the same arguments as in Theorem \ref{thm:main-norm}. In order to get the  $\Gamma$-$\liminf$ inequality,  it is sufficient to note that 
if $(u_n)\subseteq X$ is a
sequence  $L^\infty$-converging to $u$ in $X$, then  $(u_n)$  weakly $L^q$-converges to $u$ for every $q\geq 1$. As in the proof of Theorem \cite[Theorem 5.3]{EP}, we get that the sequence  $(Du_n)$ weakly converges to $Du$ in $L^{q}(\Omega,\mathbb R^d)$ for every $q>1$. 
 In particular $(u_n)$ converges weakly to $u$ in $W^{1,q}(\Omega,\mathbb R^{Nd})$ for every $q>N$. Then $(Du_n)$ generates a Young measure
$(\nu _x)_{x\in\Omega}$ such that $\nu_x(\mathbb R^{Nd})=1$  and \eqref{barycenter} holds,
for a.e. $x\in\Omega$. Thus the $\Gamma$-$\liminf$ inequality follows by applying Jensen's inequality in Definition \ref{RZdef}  and (iii) in Theorem \ref{RZthm}, in order to get \eqref{Finfinity}.  \QED

 \noindent {\bf Proof of Theorem \ref{thm:main-modular}.} The proofs follows the lines of the proof of \cite[Theorem 4.4]{BHH}, by applying Theorem \ref{thm:main-norm} instead of  \cite[Theorem 4.2]{BHH}.\QED
 
\noindent {\bf Proof of Corollary \ref{gamma6}.}
It is sufficient to argue as in the proof of Corollary \ref{gamma2} ,applying Theorem \ref{thm:main-modular}
to the sequence  $(E_n)$, defined in $X$ by equation \eqref{Ephi} (with $\phi=\phi_n$),
to get that it $\Gamma\hbox{-}$converges to $E_\infty:X \to (0,+\infty)$ with
 respect to  the $L^{\infty}$ strong  convergence. \QED
 \color{black}

As emphasized in \cite{BHH}, the assumptions of Theorem~\ref{thm:main-modular} differ from Theorem~\ref{thm:main-norm}, since the condition $\frac1c \le \phi_n(x,1) \le c$ is not sufficient 
in the latter theorem. To this end we refer to \cite[Example 4.5]{BHH}.

%
%

%

\section{$\mathbf{L}^{p(\cdot)}$- approximation of supremal functionals under no convexity assumptions}\label{pxgamma}


\color{black}
In this section we remove assumption $\mathbf{(H1)}$, i.e. the ${\rm curl}_{p >1}$- Young quasiconvexity assumption on $f:\Omega \times \mathbb R^{Nd}\to \mathbb R$, in its last variable, under the restriction that $f$ in\eqref{Fphi} and \eqref{Ephi} does not depend on the second variable, i.e. $f:\Omega \times \mathbb R^{Nd}\to [0,+\infty)$ and $f(x,\cdot)$ satisfies a growth condition  of type $|\cdot|^\gamma$, also from above, 
and assuming that $\phi_n(x,t):= t^{p_n(x)}$, i.e. the functionals appearing in \eqref{Fphi} and \eqref{Ephi} will be specialized in \eqref{approxseq} and \eqref{approxseqmod}, respectively.
Indeed, the result below 
provides an extension 
 of \cite[Theorem 2.2]{PZ}, to the case of modulars  in the variable exponent Sobolev spaces and with more general growth conditions.

On the other hand, in contrast with \cite{PZ}, we consider Carath\'eodory integrands instead of measurable ones (or normal ones as in the previous section). This is due to the exploited technical tools relying on integral representation of envelopes of integral functionals.  Indeed, the relaxation and representation results available for integral functionals with only measurable integrands in $W^{1,p}$ and $W^{1,p(\cdot)}$ (see \cite{D} and \cite{MM}, respectively)  provide a representation in terms of a generic quasiconvex function (see Definition \ref{qcxdef} below), which  in principle,  in the case $p$ constant, may differ from the one obtained in the case $p$ variable, and, in principle, no comparison among the two integrands is possible. In contrast, in our arguments to obtain a supremal representation by means $L^{p(\cdot)}$-approximation,  we rely in the case $p$ constant, passing from $p(\cdot)$ to a suitable constant exponent $p$.

We also observe that, differently from the results in Section 4, we assume Lipschitz regularity for $\partial \Omega$ and a control on the growth rates, see assumptions \eqref{pn1} and \eqref{pn2} below. Furthermore in Remarks \ref{remQinfty} and \ref{rem55} we discuss the removal of the Lipschitz assumptions and the fact that \eqref{pn2} is not necessary for Theorem \ref{relmodular1}.

Now we recall the notion of quasiconvexity  (see e.g. \cite{D}), that will be used in the sequel.
\begin{defn}\label{qcxdef}
Let $g : \mathbb{R}^{Nd} \rightarrow \mathbb{R}$ be a Borel function and let $Q :=(0; 1)^N$. Then $g$ is said quasiconvex (in the sense of Morrey)  if 
\[
g(\xi)\leq\int_Q g(\xi + D u(x))\,dx 
\]
 for every  and for every  $u \in W^{1,\infty}_0(Q;\mathbb R^d)$, $\xi\in\mathbb{R}^{Nd}$.
 \end{defn}
 
 \color{black}

\color{black}

\begin{thm}\label{relnorm1}
Let $\Omega \subset \mathbb{R}^N$ be a bounded open set with Lipschitz boundary. Let $f: \Omega \times \mathbb{R}^{Nd} \rightarrow [0, + \infty[$ be a Carath\'eodory function satisfying $(\mathbf{H2})$, such that   \beq\label{gammabove}
\exists C >0: f(x,\xi) \leq C(|\xi|^\gamma +1) \hbox{ for a.e. } x \in \Omega  \hbox{ and for every } \xi \in \mathbb R^{Nd}.
\eeq 
Let $F_n:L^{1}(\Omega,\mathbb R^{d})\to  [0, + \infty]$ be the functional defined by 
\beq\label{approxseq}
F_n(u):=\left\{\begin {array}{cl} \displaystyle  \|f(\cdot,Du)\|_{p_n(\cdot)}
&  \hbox{if } \, u\in W^{1,p_n(\cdot)}(\Omega,\mathbb R^{d})\\
+\infty  & \hbox{otherwise,}
\end{array}\right.
\eeq
where $p_n: \Omega \rightarrow [1, + \infty) $  is  a sequence of bounded variable exponents   such that 
\begin{equation}\label{pn1}
\lim_{n\to +\infty}p_n^{-}= +\infty,
\end{equation}

and 
\begin{equation}\label{pn2}
\exists \beta > 0: \frac{p_n^+}{p_n^-} \leq \beta, \hbox{ for every } n \in \mathbb N.
\end{equation}


Then $(F_n)$ sequentially $\Gamma (L^1(\Omega))$-converges with respect to the $L^1(\Omega;\mathbb R^d)$ weak topology, as $n \rightarrow \infty$ to the functional
\begin{equation*}
   F_\infty(u) = 
\begin{cases}
\|{\mathcal Q}_\infty f(\cdot, Du)\|_{\infty}, &\text{if } u \in W^{1,1}_\loc(\Omega, \mathbb R^d), 
\\
\infty, &\text{otherwise},
\end{cases}  
\end{equation*}
where
\begin{equation}\label{Qinftydef}
\mathcal{Q}_{\infty} f(x, \cdot) := \sup_{n \ge 1} (\mathcal{Q} f^n)^{1/n}(x, \cdot),
\end{equation}
being $\mathcal{Q}f^n = \mathcal{Q}(f^n)$ the quasiconvex envelope of $f^n$, i.e. 
\begin{equation}
\label{quasiconvexenvelope}
\mathcal{Q}(f^n)(x, \cdot) := \sup\{h: \mathbb{R}^{Nd} \rightarrow [0,+\infty] : h \text{ quasiconvex and } h \le f^n\}.
\end{equation}

\end{thm}


\color{black}
\begin{rem}
\label{remQinfty}

\begin{itemize}
\item[(i)] \textnormal{We underline that, in general $Q_\infty f$ is not level convex, but it is
 a strong Morrey quasiconvex (also called BJW-quasi-level convex) function, see \cite[Section 5]{PZ} for a proof.  
 We stress that this property is a consequence of the lower semiconinuity of the $\Gamma$-limit functional \eqref{Finfinity} (see \cite{BJW} and \cite{RZ}).} 

\item[(ii)] \textnormal{We observe also that $\mathcal Q_\infty$ $f$ is ${\rm curl}-\infty$ quasiconvex (see \cite{AP, AP2, RZ}), i.e.
 \begin{align*}\mathcal Q_\infty f(x,\xi)= &\lim_{p\to +\infty} \inf\left\{\left(\int_Q ({\mathcal Q}_\infty f)^p(x, \xi + D u (y))\,dy\right)^{\tfrac{1}{p}}: u \in W^{1,\infty}_{\rm per} (Q; \mathbb R^N)\right\}\\
=& \lim_{p\to +\infty} (\mathcal Q((\mathcal Q_\infty f)^p)^{\frac{1}{p}}(x,\xi),
\end{align*}in view of \cite[Proposition 4.6, Theorem 4.8]{RZ} (we also refer to \cite{AP} where the ${\rm curl}-\infty$-quasiconvexity is deduced as a consequence of $L^p$-approximation). }

\item[(iii)] \textnormal{For the sake of completeness (again referring to \cite[Theorem 4.8]{RZ}, we also remark that $Q_\infty f(x,\cdot)$ is {\rm curl}-Young-quasiconvex, i.e. it satisfies the inequality in Definition \ref{RZdef} for every $W^{1,\infty}$-gradient Young measure.}

\item[(iv)] \textnormal{We conclude this remark, recalling that we have provided a direct proof of the lower bound in Theorems \ref{relnorm1} and \ref{relmodular1}, since, besides we rely on $L^p$-approximation results of \cite{PZ}, our growth conditions differ (i.e. they are milder) from those considered therein.}

\item[(v)] \textnormal{The boundedness of $\Omega$ and the Lipschitz regularity of $\partial \Omega$ are used only in the proof of the lower bound inequality. They can be removed relying on the results of Section 4, assuming that the constant $\gamma$ in {\bf (H1)} and \eqref{gamma2} is $1$. Indeed the same growth is inherited by $\mathcal {\mathcal Q}_\infty f$, that turns out to be ${\rm curl}_{(p>1)}$-Young quasiconvex (see \cite[(7) of Theorem 4.8]{RZ} for a proof)  and such that ${\mathcal Q}_\infty f \leq f$, hence it suffices to exploit the lower bound inequality of Theorem \ref{relnorm1}.}
 \end{itemize}
 \end{rem}
 
\begin{proof}

The proof will be obtained in two steps, first by proving the lower bound estimate and then the existence of a recovery sequence, in accordance with Proposition \ref{seqcharac}.

Concerning the $\Gamma$-$\liminf$ inequality, let $(u_n) \subset L^1(\Omega; \mathbb{R}^d)$ converge weakly to $u \in L^1(\Omega; \mathbb{R}^d).$ Without loss of generality, we assume that
\[
\liminf_{n \rightarrow \infty} F_n(u_n) = \lim_{n \rightarrow \infty} F_n(u_n) = M < \infty.
\]
For this first step we rely on the results contained in \cite[Theorem 2.2]{PZ} for the case $p(x)\equiv p$ constant, indeed, a careful inspection of the proof of \cite[Theorem 2.2]{PZ}, guarantees that the latter result holds with $f$ satisfying ${\bf (H2)}$ and \eqref{gammabove}
instead of the linear growth condition therein. Furthermore, the Lipschitz regularity of $\partial \Omega$ and Sobolev and Rellich's theorems ensure the validity of the $\Gamma$-convergence result in \cite[Theorem 2.2]{PZ} also with respect to the sequential weak convergence in $L^1(\Omega;\mathbb R^d)$.

In details we observe that in view of the Lipschitz regularity of $\partial \Omega$, ${\bf( H2)}$ and  \eqref{gammabove}, for any constant $p$ sufficiently large, the functionals 

\begin{align*}
I_p(u):=\left\{
\begin{array}{ll}
\displaystyle \left(\int_\Omega f^p(x,D u(x))dx\right)^{1/p} & \hbox{ if }u \in W^{1,p}(\Omega;\mathbb R^d),\\
+ \infty &\hbox{ otherwise in } C(\overline \Omega;\mathbb R^d),
\end{array}
\right.\\
G_p(u):=\left\{
\begin{array}{ll}
\displaystyle \left(\int_\Omega f^p(x, Du(x))dx\right)^{1/p} & \hbox{ if }u \in W^{1,p\gamma}(\Omega;\mathbb R^d),\\
+ \infty &\hbox{ otherwise in } C(\overline \Omega;\mathbb R^d),
\end{array}
\right.
\\ \hbox{ and }
\\
H_p(u):=\left\{
\begin{array}{ll}
\displaystyle \left(\int_\Omega f^p(x, Du(x))dx\right)^{1/p} & \hbox{ if }u \in W^{1,1}(\Omega;\mathbb R^d),\\
+ \infty &\hbox{ otherwise in } C(\overline \Omega;\mathbb R^d),
\end{array}
\right.
\end{align*}
coincide.
For the readers' convenience we refer to \cite{EP} where similar arguments have been used to pass from one convergence to another one in the proof of their $\Gamma$-convergences results.\\
\\
At this point, for all $n \in \mathbb N$, 
let $(u_n)\subset W^{1,p_n^-}(\Omega;\mathbb R^d)$ and
\begin{eqnarray}\label{5.8BEZ}
\supess_{x \in \Omega} {\mathcal Q}_{\infty} f (\cdot, Du(\cdot)) & \le & \liminf_{n\to +\infty}\frac{1}{C_{p_n^-}}\left( \int_{\Omega} \mathcal{Q}(f^{p_n^-}(x, D u_n(x)) dx\right)^{\frac{1}{p_n^-}} \nonumber\\
&\le & \liminf_{n \rightarrow +\infty}\frac{1}{C_{p_n^-}}\left( \int_{\Omega} f^{p_n^-}(x, D u_n(x)) dx\right)^{\frac{1}{p_n^-}} 
\\
&\le & \, \liminf_{n \rightarrow \infty} \|f(\cdot, Du_n) \|_{p_n(\cdot)}. \nonumber
\end{eqnarray}
\color{black} 
where $C_{p_n^-}$ denotes the quantity $\big(2L_n(|\Omega|+c)\big)^\frac{1}{p_n^-}$ appearing in \eqref{Prop3.8} which goes to 1 when $p_n^- \rightarrow \infty.$ We remark that the first inequality is given by \cite[Theorem 2.2]{PZ}, taking into account that the $\Gamma$-convergence result stated therein with respect to the uniform convergence, holds also with respect to the sequential weak $L^1$- convergence, while the second one is obtained by using $
\mathcal{Q}(g) \le g.$ Finally the last inequality is deduced by Proposition \ref{prop:embedding} with the choices $\phi_n(x,t) = t^{p_n(x)}$ and $p = p_n^-:$ moreover we can consider $L_n = L = 1$ in \eqref{Prop3.8} as long as $\phi_n(x,\cdot)^{\frac{1}{p_n^-}}$ is convex, with $\phi_n(x,t)$ satisfying \ainc{q_n}.
taking into account that the convexity assumption applied to $t_1 = t$ and $t_2 = 0$, gives
\begin{equation}
\label{conto-conv}
\phi_n(x,\lambda t)^{\frac{1}{p_n^-}} = \phi_n(x,\lambda t_1 + (1 - \lambda) t_2)^{\frac{1}{p_n^-}} \le \lambda\phi_n(x,t_1)^{\frac{1}{p_n^-}} + (1 - \lambda)\phi_n(x,t_2)^{\frac{1}{p_n^-}} = \lambda \phi_n(x,t)^{\frac{1}{p_n^-}},
\end{equation}
which is equivalent to \ainc{q_n} with $L = 1$ uniformly in $n \in \mathbb{N}.$ 
\color{black}
This implies the $\Gamma-$liminf inequality.
\\
\\
For what concerns the upper bound, we first observe that 
it suffices to consider the case $\|\mathcal{Q}_{\infty}(f(x, D u))\|_{\infty} = 1$, since the general case can be reduced to it by 
the normalization $\tilde{v}:= \frac{v}{\|v\|_{\infty}}$ with $v:=\mathcal{Q}_{\infty}(f(x, D u))$.

In order to estimate $\inf\{\displaystyle \limsup_{n\to +\infty} \|f(x, Du_n(x))\|_{p_n(\cdot)}: u_n \rightharpoonup u \hbox{ in }L^1\}$, 
taking into account \eqref{relaztotale} we exploit the following inequality
\begin{align}\label{BEZub}
& \|f(x, Du_n(x))\|_{p_n(\cdot)} \nonumber \\
&\leq \left [ \max  \left\{  \left( \int_\Omega f^{p_n(x)}(x, Du_n(x))dx\right)^{\frac 1 {p_n^-} } , \left(\int_{\Omega}f^{p_n(x)}(x, Du_n(x))dx\right)^{\frac 1 {p_n^+} } \right \}\right ],
\end{align}
for every $n \in \mathbb N$.

At this point, 
we remark that,  since we are dealing with an upper bound estimate for $\Gamma$-convergence we are entitled to pass from sequences $u_n \rightharpoonup u$ in $L^1(\Omega;\mathbb R^d)$ to sequences $u_n \to u$ in $L^1(\Omega;\mathbb R^d)$. Furthermore by \cite[Proposition 6.11]{DM93}, and by the  fact that the right-hand side of \eqref{BEZub} can be written as $\Phi_n(\int_\Omega f^{p_n(x)} (x, D u_n(x))dx)$, with $\Phi_n:[0,+\infty) \to [0, +\infty)$, defined as
\[\Phi_n(H)= \left\{\begin{array}{ll} H^{\frac{1}{p_n^+}} \hbox{ if }  |H| \leq 1,\\
H^{\frac{1}{p_n^.}} \hbox{ if }  |H| >1\end{array}\right.\]  continuous and increasing, so that
 $\Phi_n$ commutes with the operation of taking the lower semicontinuous envelope, in view of \cite[eq. (6.3) in Proposition 6.16]{DM93},  we are allowed to deal with the relaxed functionals appearing in the right hand side of \eqref{BEZub} 
which, in view of \cite[Theorem 4.11]{MM} leads to 
\begin{align}&\inf\{\limsup_n\|f(x, D u_n(x))\|_{p_n(\cdot)}: u_n \to u \hbox{ in } L^1(\Omega;\mathbb R^d)\} \label{limsupest}\\
\le &\inf\left\{\limsup_n \left [ \max  \left\{  \left( \int_\Omega   f^{p_n(x)}(x, Du_n(x))dx\right)^{\frac 1 {p_n^-} } , \left(\int_{\Omega}f^{p_n(x)}(x, Du_n(x))dx\right)^{\frac 1 {p_n^+} } \right \} \right ]: u_n \to u \hbox{ in } L^1(\Omega;\mathbb R^d)\right\}\nonumber\\
= &\inf\left\{\limsup_n \left [ \max  \left\{  \left( \int_\Omega \mathcal Q( f^{p_n(x)})(x, Du_n(x))dx\right)^{\frac 1 {p_n^-} }, \left(\int_{\Omega}\mathcal Q(f^{p_n(x)})(x, Du_n(x))dx\right)^{\frac 1 {p_n^+} } \right\}  \right ]: u_n \to u \hbox{ in } L^1(\Omega;\mathbb R^d)\right\} \nonumber
\end{align}
where $\mathcal{Q}(f^{p_n(x)})(x,\cdot)$ is defined in \eqref{quasiconvexenvelope}.
\\
\\

On the other hand, \cite[Remark 5.1]{PZ} and \eqref{Qinftydef} ensure that 
\[
\mathcal Q (f^{p_n(x)})(x, Du(x))\leq (\mathcal Q_\infty f)^{p_n(x)}(x, Du(x)),
\] 
hence, it suffices to consider  for any sequence $(u_n)$ converging to u in $L^1(\Omega;\mathbb R^d) $, 
\begin{align*}
\limsup_n \left [ \max  \left\{  \left( \int_\Omega \left((\mathcal Q_\infty f)^{p_n(x)}(x, Du_n(x))\right)dx\right)^{\frac 1 {p_n^-} } , \left(\int_{\Omega}\left((\mathcal Q_\infty  f)^{p_n(x)}(x, Du_n(x))\right)dx\right)^{\frac 1 {p_n^+} } \right\} \right ].
\end{align*}

\color{black}
Since $\|\mathcal{Q}_{\infty}(f(x, D u))\|_{\infty} = 1$, we
have that 
\begin{equation}\label{3.4BEZ}\int_\Omega (\mathcal Q_\infty f)^{p_n^+}(x, D u(x))dx\leq \int_{\Omega}(\mathcal Q_\infty  f)^{p_n(x)}(x, Du(x))dx \leq \mathcal L^N(\Omega).\end{equation}

 Following \cite[(3.4), (3.5) and (3.6)]{EP}, for every $n \in \mathbb N$, 
and, thanks to \eqref{3.4BEZ}, we get that
\begin{align*}1 &= \lim_{n\to +\infty}
(\mathcal L^N(\Omega))^{1/p_n^+}
\geq \limsup_{n\to +\infty}
\left(\int_\Omega ({\mathcal Q_\infty} f)^{p_n(x)}(x, Du (x))dx\right)^{\frac{1}{p_n^+}} \\
& \geq \limsup_{n\to +\infty}
\left(\int_\Omega ({\mathcal Q_\infty} f)^{p_n^+}(x, Du (x))dx\right)^{\frac{1}{p_n^+}} = \|{\mathcal Q_\infty} f(x, D u(x))\|_{\infty}= 1.
\end{align*}
Observing that we can replace $\limsup$
by $\liminf$ we can conclude that the $\limsup$ is a limit and
\begin{equation}\label{prima}
 \lim_{n\to +\infty}
\left(\int_\Omega ({\mathcal Q_\infty} f)^{p_n(x)}(x, Du (x))dx\right)^{\frac{1}{p_n^+}}=1.
\end{equation}


By \eqref{pn2}, we have that $1 \leq \beta_n:= \frac{p_n^+}{p_n^-} \leq \beta$ and 

\begin{equation}\label{seconda}
 \lim_{n\to +\infty}
\left(\int_\Omega ({\mathcal Q_\infty} f)^{p_n(x)}(x, Du (x))dx\right)^{\frac{1}{p_n^-}}= \lim_{n\to +\infty}
\left(\int_\Omega ({\mathcal Q_\infty} f)^{p_n(x)}(x, Du (x))dx\right)^{\frac{\beta_n}{p_n^+}}=1.
\end{equation}
 By \eqref{prima}
 and \eqref{seconda}, 
 we can conclude that
 \begin{align*}\limsup_{n\to +\infty }\max  \left\{  \left( \int_\Omega \left((\mathcal Q_\infty f)^{p_n(x)}(x, Du(x))\right)dx\right)^{\frac 1 {p_n^-} } , \left(\int_{\Omega}\left((\mathcal Q_\infty  f)^{p_n(x)}(x, Du(x))\right)dx\right)^{\frac 1 {p_n^+} } \right\} \\
 =\lim_{n\to +\infty} \max  \left\{  \left( \int_\Omega \left((\mathcal Q_\infty f)^{p_n(x)}(x, Du(x))\right)dx\right)^{\frac 1 {p_n^-} } , \left(\int_{\Omega}\left((\mathcal Q_\infty  f)^{p_n(x)}(x, Du(x))\right)dx\right)^{\frac 1 {p_n^+} } \right\}=1,
 \end{align*}
which, together with \eqref{limsupest}, 
concludes the proof of the upper bound and proves our statement.
\color{black}

\end{proof}

\begin{thm}\label{relmodular1}
Let $\Omega \subset \mathbb{R}^N$ be a bounded open set with Lipschitz boundary. Let $f: \Omega \times \mathbb{R}^{Nd} \rightarrow [0, + \infty[$ be a Carath\'eodory function satisfying $(\mathbf{H2})$ and \eqref{gammabove}.
\\
Let $E_n:L^{1}(\Omega,\mathbb{R}^{d})\to  [0, + \infty]$ be the functional defined by 
\beq\label{approxseqmod}
E_n(u):=\left\{\begin {array}{cl} \displaystyle \ \int_{\Omega} \frac{1}{p_n(x)}  f^{p_n(x)}(x,Du(x))dx
&  \hbox{if } \, u\in W^{1,p_n(\cdot)}(\Omega,\mathbb{R}^{d})\\
+\infty  & \hbox{otherwise,}
\end{array}\right.
\eeq
where $p_n: \Omega \rightarrow [1, + \infty) $  is  a sequence of bounded variable exponents   satisfying \eqref{pn1}. 
\\
Then 
$(E_n)$ $\Gamma$ converges with respect to the $L^1(\Omega;\mathbb R^d)$ weak convergence, as $n \rightarrow \infty$ to the functional  
\begin{equation}
\label{Einfinityinfinity}
  E_\infty(u) = 
\begin{cases}
0, &\text{if } u \in W^{1,1}_\loc(\Omega, \mathbb R^d), \text{ and }|{\mathcal Q}_\infty f(\cdot, Du)|\le 1 \text{ a.e.} \\
\infty, &\text{otherwise},
\end{cases}  
\end{equation}
where $\mathcal{Q}_{\infty} f(x,\cdot)$ has been introduced in \eqref{Qinftydef}.
\end{thm}

\begin{proof}
Concerning the $\Gamma-$liminf inequality, let $(u_n) \subset L^1(\Omega; \mathbb{R}^d)$ converge weakly to $u \in L^1(\Omega; \mathbb{R}^d).$ Without loss of generality, we assume that
\[
\liminf_{n \rightarrow \infty} E_n(u_n) = \lim_{n \rightarrow \infty} E_n(u_n) = M < \infty.
\]
Define
\[
\varphi_n(x,t) := \frac{t^{p_n(x)}}{p_n(x)} 
\]
By the right hand side of the inequality contained in \cite[Corollary 3.2.10]{HarH19}, with $\varphi$  which satisfy condition $(aInc)_{p_n^-}$ and with the constant $a$ replaced by the constant $L_n$ given by condition $(aInc)_{p_n^-}$, see 
\eqref{conto-conv},
with $\phi_n(x, t) = {\varphi_n}(x,t)$ and $q_n = p_n^-$, we get
\[
\|f(\cdot, Du_n) \|_{\varphi_n} \le \, \max \left \{ \left (\int_{\Omega} \frac{1}{p_n(x)} f^{p_n(x)}(x, D u_n(x)) \, dx \right )^{1/p_n^-}, 1 \right \}
\]
and hence
\[
\limsup_{n \rightarrow \infty} \|f(\cdot, Du_n) \|_{\varphi_n}  \le \, 1.
\]
For all $n \in \mathbb N$, 
let $(u_n)\subset W^{1,p_n^-}(\Omega;\mathbb R^d)$.

By \eqref{5.8BEZ}, we get
\begin{eqnarray}
\supess_{x \in \Omega} {\mathcal Q}_{\infty} f (\cdot, Du(\cdot)) 
\leq \liminf_{n \rightarrow \infty} \|f(\cdot, Du_n) \|_{\varphi_n}  \le \, 1. \nonumber
\end{eqnarray}
\\

In this way we are able to conclude that
\[ \esssup_{x \in \Omega} {\mathcal Q}_{\infty}f(x, Du(x))\le 1,\] which proves the lower bound taking into account  definition \eqref{Einfinityinfinity}.

Now, we prove the $\Gamma$-$\limsup$ inequality.\\
Without loss of generality we can assume that $E_{\infty}(u) < \infty$, otherwise the thesis is trivial. 
 
%

Arguing as in the proof of Theorem \ref{relnorm1}, by \cite[Proposition 6.11]{DM93}, we can assume that the functionals $E_n$ under consideration coincide with their relaxed ones, i.e. we can assume, in view of \cite[Theorem 4.11]{MM}, that $E_n(u)=\int_{\Omega} \frac{1}{p_n(x)} \mathcal{Q} \left ( f^{p_n(x)} (x, D u(x)) \right ) dx  $ for every $u \in W^{1, p_n(\cdot)}(\Omega;\mathbb R^d)$. 
 Since $\|{\mathcal Q}_{\infty}f(\cdot, Du)\|_{\infty} \le 1$, recalling that for a.e. $x \in \Omega$ and for every $\xi \in \mathbb R^{Nd}$, $\mathcal{Q}(f^p)^{\frac{1}{p}}(x,\xi)$ is increasing in $p$, \\thanks to the definition of $\mathcal{Q}_{\infty}$ \eqref{Qinftydef}, we have 
\[
\limsup_{n \rightarrow \infty} \int_{\Omega} \frac{1}{p_n(x)} \mathcal{Q} \left ( f^{p_n(x)} (x, D u(x)) \right ) dx \le \limsup_{n \rightarrow \infty} \int_{\Omega} \frac{1}{p_n^-}  (\mathcal{Q}_{\infty}f)^{p_n(x)} (x, D u(x)) \, dx = 0,
\]
and thus the desired inequality follows and the proof is concluded.
\end{proof}

\begin{rem}\label{rem55}

\textnormal{Observe that in both Theorems \ref{relnorm1} and \ref{relmodular1}, the proof of the $\Gamma$-$\liminf$ inequality holds for more general $\phi_n(x,t)$, not necessarily of $p_n$-power type, provided that $(aInc)_{p_n}$ holds for the same constant $L$ for all $n$, without imposing \eqref{gammabove},\eqref{pn1} and \eqref{pn2}.}

\textnormal{In view of this last observation, in the case $\Omega$ has Lipschitz boundary, when $f:\Omega \times \mathbb R^{Nd}\to \mathbb R$, as a consequence of Theorems \ref{relnorm1} and \ref{relmodular1}, in view of (iv) in Remark \ref{remQinfty}, and taking into account that the proofs of the upper-bound inequalities in Theorems \ref{thm:main-norm}, \ref{thm:main-modular}, Corollaries \ref{gamma2} and \ref{gamma6} (for which it is sufficient to consider a constant recovery sequence), we could replace in Theorems \ref{thm:main-norm}, \ref{thm:main-modular}, Corollaries \ref{gamma2} and \ref{gamma6} the ${\rm curl}_{(p>1)}$-Young quasiconvexity of $f(x,\cdot)$ by  ${\rm curl}_\infty$-quasiconvexity. This latter assumption is in general weaker than the previous one, see \cite{RZ}.}

\smallskip
\noindent \textnormal{The same observations regarding the Lipschitz continuity of $\partial \Omega$ contained in Remark \eqref{remQinfty}(v), apply to Theorem \ref{relmodular1}.}

\textnormal{We also remark that \eqref{pn1} is not been exploited in the proof of this theorem, in contrast to Theorem \ref{relnorm1}, where it has been used in the proof of the upper bound inequality.}
\end{rem}

{\noindent{\bf Acknowledgments.}
G. B. and E. Z. acknowledge partial support by the INdAM - GNAMPA Project``Composite materials and microstructures''. M. E.  has been partially supported by PRIN 2020 ``Mathematics for industry 4.0 (Math4I4)'' (coordinator P. Ciarletta).
\color{black}
E. Z. acknowledges the received support through Sapienza Progetti d'Ateneo 2022 piccoli ``Asymptotic Analysis for composites, fractured materials and with defects '', and through PRIN 2022 `` Mathematical Modelling of Heterogeneous Systems'' (coordinator E. N. M. Cirillo), CUP B53D23009360006.
The authors were partially supported by the Gruppo Nazionale per l'Analisi Matematica, la Probabilità e le loro Applicazioni (GNAMPA) of the Istituto Nazionale di Alta Matematica (INdAM)  also  through the INdAM - GNAMPA Project ``Prospettive nelle scienze dei materiali: modelli variazionali, analisi asintotica ed omogeneizzazione'', CUP E53C23001670001.

\end{document}